\documentclass{amsart}
\usepackage{pgf,amsthm,amscd,amssymb,verbatim,epsf,amsmath,eurosym,amsfonts,mathrsfs,graphicx,tikz-cd,lscape}
\usepackage{amsthm,amscd,amssymb,verbatim,epsf,amsmath,amsfonts,mathrsfs,graphicx,dirtytalk,pdfpages,mdframed,tikz-cd,xy}
\usepackage[colorlinks=true,linkcolor=blue,citecolor=blue]{hyperref}
\usepackage[font=small,labelfont=bf]{caption} 
\usepackage{epigraph}
\begin{document}
\theoremstyle{plain}
\newtheorem{Thm}{Theorem}
\newtheorem{Cor}[Thm]{Cor}
\newtheorem{Ex}[Thm]{Example}
\newtheorem{Con}[Thm]{Conjecture}
\newtheorem{Main}{Main Theorem}
\newtheorem{Lem}[Thm]{Lemma}
\newtheorem{Prop}[Thm]{Proposition}

\theoremstyle{definition}
\newtheorem{Def}[Thm]{Definition}
\newtheorem{Note}[Thm]{Note}
\newtheorem{Question}[Thm]{Question}

\theoremstyle{remark}
\newtheorem{notation}[Thm]{Notation}
\renewcommand{\thenotation}{}

\errorcontextlines=0
\renewcommand{\rm}{\normalshape}%
\newcommand{\transv}{\mathrel{\text{\tpitchfork}}}
\makeatletter
\newcommand{\tpitchfork}{%
  \vbox{
    \baselineskip\z@skip
    \lineskip-.52ex
    \lineskiplimit\maxdimen
    \m@th
    \ialign{##\crcr\hidewidth\smash{$-$}\hidewidth\crcr$\pitchfork$\crcr}
  }%
}
\makeatother

\title[From CT scans to 4-manifold topology]{From CT scans to 4-manifold topology \\via neutral geometry}

\author{Brendan Guilfoyle}
\address{Brendan Guilfoyle\\
          School of Science Technology Engineering and Mathematics\\
          Munster Technological University, Kerry\\
          Tralee\\
          Co. Kerry\\
          Ireland.}
\email{brendan.guilfoyle@mtu.ie}

\begin{abstract}
In this survey paper the ultrahyperbolic equation in dimension four is discussed from a geometric, analytic and topological point of view. The geometry centres on the canonical neutral metric on the space of oriented geodesics of 3-dimensional space-forms, the analysis discusses a mean value theorem for solutions of the equation and presents a new solution of the Cauchy problem over a certain family of null hypersurfaces, while the topology relates to generalizations of codimension two foliations of 4-manifolds.
\end{abstract} 

\date{\today}
\maketitle

\setlength{\epigraphwidth}{0.7\textwidth}
\epigraph{\it The air is full of an infinity of straight lines and rays \\
which cut across each other without displacing each other and \\
which reproduce on whatever they encounter \\
the true form of their cause.}{{ Leonardo da Vinci\\MS. A. 2v, 1490}}

\section{Introduction}

Our staring point is, as the title suggests, the acquisition of density profiles of biological systems using the loss of intensity experienced by a ray traversing the system. Basic mathematical physics arguments imply that this loss is modelled by the integral of the density function along the ray. One goal of Computerized Tomography is to invert the X-ray transform: reconstruct a real-valued function on ${\mathbb R}^3$ from its integrals over families of lines. 

The reconstruction of a function on the plane from its value on all lines, or more generally, a function on Euclidean space from its value on all hyperplanes, dates back at least to Johann Radon \cite{Radon}. One could argue that Allan MacLeod Cormack's 1979 Nobel prize for the theoretical results behind CAT scans \cite{Cormack} is the closest that mathematics has come to winning a Nobel prize, albeit in Medicine. The choice of axial rays reduces the inversion of the X-ray transform to that of the Radon transform over planes in ${\mathbb R}^3$ \cite{Hounsfield73}.

The basic problems of tomography - acquisition and reconstruction - arise far more widely than just medical diagnostics, finding application in industry \cite{Zhang}, geology \cite{Tonai}, archaeology \cite{Payne} and transport security \cite{Mouton}. Indeed, advances in CT technology, trialed in Shannon Airport recently, could warrant the removal of the 100ml liquid rule for airplane travellers globally \cite{Russell}.

Rather surprisingly, sitting behind the X-ray transform and its many applications is a largely unstudied second order differential equation: the ultrahyperbolic equation. For a function $u$ of four variables $(X_1,X_2,X_3,X_4)$ the equation is
\begin{equation}\label{e:uhe}
\frac{\partial^2u}{\partial X_1^2}+\frac{\partial^2u}{\partial X_2^2}-\frac{\partial^2u}{\partial X_3^2}-\frac{\partial^2u}{\partial X_4^2}=0.
\end{equation}
The reasons for the relative paucity of mathematical research on the equation despite the link to tomography will be discussed below. 

The purpose of this mainly expository paper is to describe recent research on the ultrahyperbolic equation, its geometric context and its applications. It turns out that the ultrahyperbolic equation is best viewed in terms of a conformal class of neutral metrics and that in this context it advances new paradigms that can contribute to the understanding of four dimensional topology. We now discuss the mathematical background of this undertaking before giving a more detailed summary of the paper.

\vspace{0.1in}

\subsection{Background}
The {\em X-ray transform} of a real valued function on ${\mathbb R}^3$ is defined by taking its integral over (affine) lines of ${\mathbb R}^3$. That is, given a real function $f:{\mathbb R}^3\rightarrow{\mathbb R}$ and a line $\gamma$ in ${\mathbb R}^3$, let
\[
u_f(\gamma) = \int_\gamma fdr,
\]
where $dr$ is the unit line element induced on $\gamma$ by the Euclidean metric on ${\mathbb R}^3$.

Thus we can view the X-ray transform of a function $f$ (with appropriate behaviour at infinity) as a map $u_f:{\mathbb L}({\mathbb R}^3)\rightarrow{\mathbb R}:\gamma\mapsto u_f(\gamma)$, where ${\mathbb L}({\mathbb R}^3)$, or ${\mathbb L}$ for short, is the space of oriented lines in ${\mathbb R}^3$. Here we pick an orientation on the line to simplify later local constructions, much as Leonardo does when invoking \textit{rays} as distinct from lines, and note that the space ${\mathbb L}$ double covers the space of lines. 

In comparison, the {\em Radon transform} takes a real-valued function on ${\mathbb R}^3$ and integrates it over \textit{planes} in ${\mathbb R}^3$. By elementary considerations, the space of affine planes in ${\mathbb R}^3$ is three dimensional, equal to the dimension of the underlying space, while the space of oriented lines is four dimensional. 

Thus, by dimension count, if we consider the problem of inverting the two transforms, given a function on planes one can reconstruct the original function on ${\mathbb R}^3$, while the problem is over-determined for functions on lines. The consistency condition for a function on line space to come from an integral of a function on ${\mathbb R}^3$ is exactly the ultrahyperbolic equation \cite{fjohn}.  

Viewed simply as a partial differential equation, equation (\ref{e:uhe}) is neither elliptic nor hyperbolic, and so many standard techniques of partial differential equation do not apply. Indeed, in early editions of their influential classic {\it Methods of Mathematical Physics}, Richard Courant and David Hilbert showed that the ultrahyperbolic equation in ${\mathbb R}^{2,2}$ has an ill-posed Cauchy boundary value problem when the boundary has Lorentz signature, thus relegating the equation as unphysical in a mechanical sense. 

It was Fritz John who in 1937 proved that, to the contrary, the ultrahyperbolic equation can have a well-posed characteristic boundary value problem if the boundary 3-manifold is assumed to be \textit{null}, rather than Lorentz \cite{fjohn}. Later editions of Courant and Hilbert's book acknowledge John's contribution and his discovery of the link to line space, but study of the ultrahyperbolic equation never took off in the way that it did for elliptic and hyperbolic equations.

On the other hand, by reducing the X-ray transform to the Radon transform for certain null configurations of lines,  Cormack side-stepped the ultrahyperbolic equation altogether. Moreover, for applied mathematicians, the equation, or its associated John's equations, arises mainly as a compatibility condition if more than a 3-manifold's worth of data is acquired. Its possible utility from that perspective therefore is to check such excess data, rather than to help reconstruct the function.

Our first goal, contained in Section \ref{s:2} is the geometrization of the ultrahyperbolic equation. In particular, we view it as the Laplace equation of the canonical metric ${\mathbb G}$ of signature $(++--)$ on the space ${\mathbb L}$ of oriented lines in ${\mathbb R}^3$ \cite{GK05}. The fact that ${\mathbb G}$ is conformally flat and has zero scalar curvature means that a conformal multiple of a harmonic function satisfies the flat ultrahyperbolic equation (\ref{e:uhe}).  Fritz John did not explicitly use the neutral metric, but at the cost of the introduction of unmotivated multiplicative factors in calculations, factors that can now be related with the conformal factor of the metric. 

The introduction of the neutral metric not only clarifies the ultrahyperbolic equation, but it highlights the role of the conformal group in tomography. Properties such as conformal flatness of a metric, zero distance between points or nullity of a hypersurface are properties of the conformal class of a metric. Moreover, mathematical results can be extended by applying conformal maps \cite{CG22b}. 

Section \ref{s:2} describes how these neutral conformal structures arise in the space of oriented geodesics of any 3-dimensional space-form, namely ${\mathbb R}^3$, ${\mathbb S}^3$ and ${\mathbb H}^3$. The commonality between these three spaces allows one to apply many of the results (mean value theorem, doubly ruled surfaces, null boundary problems) to non-flat spaces. Surprisingly, electrical impedance tomography calls for negative curvature and so tomography in hyperbolic 3-space is not quite as fanciful as it may at first seem - see \cite{Berenstein}. The link between the ultrahyperbolic equation and the neutral metric on the space of oriented geodesics in ${\mathbb H}^3$ as given in Theorem \ref{p:lh3} is new and so the full proof is given below. 

In Section \ref{s:3} conformal methods are used to extend both a classical mean value theorem and its interpretation in terms of doubly ruled surfaces in ${\mathbb R}^3$. Aside from the discussion of the conformal extension of the mean value theorem, the section contains a new geometric formula for a solution of the ultrahyperbolic equation given only  values on the null hypersurface formed by lines parallel to a fixed plane. In fact, this example was considered by John, but the geometric version we present using the null cone of the neutral metric has not appeared elsewhere. 

The final Section turns to global aspects of complex points on Lagrangian surfaces in ${\mathbb L}$ and an associated boundary value problem for the Cauchy-Riemann operator. This proof of the Carath\'eodory Conjecture using the canonical neutral metric on the space of oriented lines \cite{GK11} is under review,  but significant parts of the arguments have now appeared in print. In particular, the essence as to {\em why} the Conjecture is true - namely the size of the Euclidean group - has been established \cite{GK20a} and shown to be sharp \cite{G20}.

The efficacy of second order methods of parabolic partial differentiation in higher codimension has also been proven in this context for both interior \cite{GK19a} and boundary problems \cite{GK23a}. The final argument hinges on the technical point as to whether a hyperbolic angle condition in codimension two in dimension four can be made \textit{sticky} enough to confine the boundary of a line congruence evolving under mean curvature flow. This is the sole remaining part of the proof under review.

Having established the \textit{why}, this approach to the Carath\'edory Conjecture also lends itself to other independent methods of completion - one needs only to establish the existence of enough holomorphic discs attached to a given Lagrangian surface and the Conjecture follows. Indeed, a local index bound \cite{GK12} and a conjecture of Toponogov \cite{GK20b} would also follow from existence of such families. This could be proven, for example, by the use of the method of continuity and pseudo-holomorphic curves \cite{Gromov85}, which would be a first order rather than second order proof. In any event, the acceptance that this infamous Conjecture has been finally put to rest will probably only come about when it has been proven at least twice.

A positive outcome of these developments has been the first application of differential geometry in the theory of complex polynomial: the index bound for an isolated umbilic point on a real analytic surface has been shown to restrict the number of zeros inside the unit circle for a polynomial with self-inversive second derivative \cite{GK23b}. This and related issues are discussed in more detail in Section \ref{s:4}.

The reason codimension two has a special significance in four dimensional topology is briefly discussed and the final section considers topological obstructions to neutral metrics as applied to closed 4-manifolds. In the case where the 4-manifold is compact with boundary, many open questions remain about what geometric information from a neutral metric can be seen at the boundary. Whether for a neutral 4-manifold with null boundary, coming full circle, it is possible to X-ray the inside and explore its topology.

\vspace{0.1in}
\section{The Geometry of Neutral Metrics}\label{s:2}

This section discusses the geometry of metrics of indefinite signature $(++--)$. While the study of positive definite metrics and Lorentz metrics are very well-developed, the neutral signature case is less well understood, even in dimension four. Rather than the general theory, of which \cite{Davidov} is a good survey, the section will focus on spaces of geodesics and the invariant neutral structures associated with them.

\vspace{0.1in}
\subsection{The Space of Oriented Lines}\label{s:2.1}

The space ${\mathbb L}$ of oriented lines (or \textit{rays}) of Euclidean ${\mathbb R}^3$ can be identified with the set of tangent vectors of ${\mathbb S}^2$ by noting that
\begin{equation}\label{d:l}
{\mathbb L}=\{ \vec{U},\vec{V}\in{\mathbb R}^3\;|\;\; |\vec{U}|=1\; {\mbox{ and }} \;\vec{U}\cdot\vec{V}=0\;\}=T{\mathbb S}^2,
\end{equation}

where $\vec{U}$ is the direction vector of the line and $\vec{V}$ the perpendicular distance vector to the origin. 

Topologically, ${\mathbb L}$ is a non-compact simply connected 4-manifold which can be viewed as the two dimensional vector bundle over ${\mathbb S}^2$ with Euler number two. One can see the Euler number by taking the zero section, which is the 2-sphere of oriented lines through the origin and perturbing it to another sphere of oriented lines (the oriented lines through a nearby point, for example). The two spheres are easily seen to intersect in two oriented lines, hence the Euler number of the bundle is two.

This space comes with a natural projection map $\pi:{\mathbb L}\rightarrow{\mathbb S}^2$ which takes an oriented line to its unit direction vector $\vec{U}$. In fact, there is a wealth of canonical geometric structures on ${\mathbb L}$, where canonical means invariant under the Euclidean group. These include a neutral K\"ahler structure, a fibre metric and an almost paracomplex structure. All three have a role to play in what follows and so we take some time to describe them in detail.

To start with the K\"ahler metric on ${\mathbb L}$, one has
\vspace{0.1in}
\begin{Thm}\cite{GK05}\label{t:gk}
The space ${\mathbb L}$ of oriented lines of ${\mathbb R}^3$ admits a metric ${\mathbb G}$ that is invariant under the Euclidean group acting on lines. The metric is of neutral signature $(++--)$, is conformally flat and scalar flat, but not Einstein. 

It can be supplemented by a complex structure $ J_0$ and symplectic structure $\omega$, so that $({\mathbb L},{\mathbb G}, J_0,\omega)$ is a neutral K\"ahler 4-manifold.
\end{Thm}
\vspace{0.1in}
Here the complex structure $J_0$ is defined at a point $\gamma\in{\mathbb L}$ by rotation through 90$^o$ about the oriented line $\gamma$. This structure was considered in a modern context first by Nigel Hitchin \cite{Hitchin82}, who dated it back at least to Karl Weierstrass in 1866 \cite{Weier66}. 

The symplectic structure $\omega$ is by definition a non-degenerate closed 2-form on ${\mathbb L}=T{\mathbb S}^2$, and it can be obtained by pulling back the canonical symplectic structure on the cotangent bundle $T^*{\mathbb S}^2$ by the round metric on ${\mathbb S}^2$. 

These two structures are invariant under Euclidean motions acting on line space and fit nicely together in the sense that $
\omega(J\cdot,J\cdot)=\omega(\cdot,\cdot)$. The metric obtained by their composition $ {\mathbb G}(\cdot,\cdot)=\omega(J\cdot,\cdot)$, is of neutral signature $(++--)$, however. The existence of a Euclidean invariant metric of this signature on line space was first noted by Eduard Study in 1891 \cite{Study91}, but its neutral K\"ahler nature wasn't discovered until 2005 \cite{GK05}. Interestingly, the space of oriented lines in Euclidean ${\mathbb R}^n$ admits an invariant metric iff $n=3$, and in this dimension it is pretty much unique \cite{Salvai05}. This accident of low dimensions offers an alternative geometric framework to investigate the semi-direct nature of the Euclidean group in dimension three, one which expresses three dimensional Euclidean quantities in terms of neutral geometric quantities in four dimensions. 

This is but one of the many accidents that arise in the classification of invariant symplectic structures, (para)complex structures, pseudo-Riemannian metrics and (para)K\"ahler structures on the space of oriented geodesics of a simply connected pseudo-Riemannian space of constant curvature or a rank one Riemannian symmetric space \cite{AGK11}. 

Returning to oriented line space, the neutral metric ${\mathbb G}$ at a point $\gamma\in{\mathbb L}$ can be interpreted as the angular velocity of any line near $\gamma$. If the angular velocity is zero - and hence the oriented lines are null-separated - then the lines either intersect or are parallel. One can adopt the projective view, which arises quite naturally, that parallel lines intersect at infinity, and then nullity of a curve with respect to the neutral metric implies the intersection of the underlying infinitesimal lines in ${\mathbb R}^3$. Nullity for higher dimensional submanifolds will be discussed in the next section. 

The invariant neutral metric is not flat, although its scalar curvature is zero and its conformal curvature vanishes. The non-zero Ricci tensor has zero neutral length, but its interpretation in terms of a recognisable energy momentum tensor is lacking. Given the difference of signature to Lorentz spacetime, it is also difficult to see the usual physical connection as in general relativity. 

Since the metric is conformally flat, there exist local coordinates $(X_1,X_2,X_3,X_4)$ and a strictly positive function $\Omega$  so that it can be written as
\begin{equation}\label{e:cflcoords}
ds^2=\Omega^2(dX_1^2+dX_2^2-dX_3^2-dX_4^2).
\end{equation}
Such a metric has zero scalar curvature iff $\Omega$ satisfies the ultrahyperbolic equation, thus characterising a Yamabe-type problem for neutral metrics \cite{Lee87}.  Such coordinates $(X_1,X_2,X_3,X_4)$ were first constructed  using the Pl\"ucker embedding on the space of lines by John \cite{fjohn}, who showed that the compatibility condition for a function on line space to be the integral of a function on ${\mathbb R}^3$ is exactly the flat ultrahyperbolic equation in these coordinates.

Write ${\mathbb R}^{2,2}$ for ${\mathbb R}^{4}$ endowed with the flat neutral metric. In Section \ref{s:3} the ultrahyperbolic equation will be considered in more detail and an explicit formula presented for data prescribed on a certain null hypersurface.

A peculiarity of neutral signature metrics in dimension four is the existence of 2-planes on which the induced metric is identically zero, so-called {\em totally null} 2-planes. In ${\mathbb R}^{2,2}$ there are a disjoint union of two $S^1$'s worth of totally null 2-planes, termed $\alpha-$planes and $\beta-$planes.

One way to see these is to consider the null cone ${\mathcal C}_0$ at the origin. This is a cone over the 2-torus $S^1\times S^1$ given by 
\[
X_1^2+X_2^2-X_3^2-X_4^2=0.
\]
An $\alpha-$plane is a cone over a diagonal in the torus $t\mapsto (X_1+iX_2,X_3+iX_4)=(e^{it},e^{i(t+t_0)})$, while a $\beta-$plane is a cone over an anti-diagonal in the torus $t\mapsto(X_1+iX_2,X_3+iX_4)=(e^{it},e^{-i(t+t_0)})$.

This null structure exists in the tangent space at a point in any neutral four manifold and if one can piece it together in a geometric way there can be global topological consequences. One natural question is whether the $\alpha-$planes or $\beta-$plane fields are integrable in the sense of Frobenius, thus having surfaces to which the plane fields are tangent. These are guaranteed for the invariant neutral metrics endowed on the space of oriented geodesics of any 3-dimensional space-form, as they are all conformally flat \cite{DW08}. 

Roughly speaking, an $\alpha-$surface in a geodesic space is the set of oriented geodesics through a fixed point, while $\beta-$surfaces are the oriented geodesics contained in a fixed totally geodesic surface in the ambient 3-manifold. Thus a neutral metric on a geodesic space allows for the geometrization of both intersection and containment.

Restricting our attention to ${\mathbb R}^3$, the $\alpha-$planes in ${\mathbb L}$ are the oriented lines through a point or the oriented lines with the same fixed direction. The latter are the 2-dimensional fibres of the canonical projection $\pi:{\mathbb L}\rightarrow {\mathbb S}^2$ taking an oriented line to its direction.  

The distance between parallel lines in ${\mathbb R}^3$ induces a fibre metric on $\pi^{-1}(p)$ for $p\in{\mathbb S}^2$. If $\xi$ is a complex coordinates about the North pole of ${\mathbb S}^2$ given by stereographic projection and $\eta$ the complex fibre coordinate in the projection $T{\mathbb S}^2\rightarrow{\mathbb S}^2$, then the fibre metric has the form 
\begin{equation}\label{e:fibmet}
d\tilde{s}^2=\frac{4d\eta\;d\bar{\eta}}{(1+\xi\bar{\xi})^2}.
\end{equation}
In Section \ref{s:3.3} this arises in the X-ray transform from certain null data. 

Note that the complex coordinates $(\xi,\eta)$ on ${\mathbb L}$ are essentially the vectors $U$ and $V$ in definition (\ref{d:l}), the direction and perpendicular distance to the origin. They are related to John's conformal flat coordinates  $(X_1, X_2, X_3,X_4)$ by
\vspace{0.1in}
\begin{Prop}\label{p:conf}\cite{CG22a}
For complex coordinates $(\xi,\eta)$ on $T{\mathbb S}^2$, over the upper hemisphere $|\xi|^2<1$ the conformal coordinates $(X_1, X_2, X_3,X_4)$ are
\[
X_1+iX_2=\frac{2}{1-\xi^2\bar{\xi}^2}\left(\eta+\xi^2\bar{\eta}-i(1+\xi\bar{\xi})\xi\right)
\]
\[
X_3+iX_4=\frac{2}{1-\xi^2\bar{\xi}^2}\left(\eta+\xi^2\bar{\eta}+i(1+\xi\bar{\xi})\xi\right).
\]
\end{Prop}
\vspace{0.1in}

We turn now to null 3-manifolds (or hypersurfaces) in a neutral 4-manifold. An example of such is the null cone of a point in ${\mathbb L}$. Fix any oriented line $\gamma_0\in {\mathbb L}$ and define its {\it null cone} to be
\[
C_0(\gamma_0)=\{\gamma\in {\mathbb L}\;|\;\;Q(\gamma_0,\gamma)=0\},
\]
where $Q$ is the neutral distance function introduced by John \cite{fjohn}. For convenience introduce the complex conformal coordinates given in terms of the real conformal coordinates of equation (\ref{e:cflcoords}) by
\[
Z_1=X_1+iX_2\qquad\qquad Z_2=X_3+iX_4.
\]
If two oriented lines $\gamma,\tilde{\gamma}$ have complex conformal coordinates $(Z_1,Z_2)$ and  $(\tilde{Z}_1,\tilde{Z}_2)$ then the neutral distance function is
\[
Q(\gamma,\tilde{\gamma})=|Z_1-\tilde{Z}_1|^2-|Z_2-\tilde{Z}_2|^2.
\]
Two oriented lines have zero neutral distance iff either they are parallel or they intersect. The null cone arises in the formula for the ultrahyperbolic equation in Theorem \ref{t:uheinv}.

More generally, null hypersurfaces in ${\mathbb L}$ can be understood as 3-parameter families of oriented lines in ${\mathbb R}^3$ as follows. The degenerate hyperbolic metric induced on a null hypersurface ${\mathcal H}$ at a point $\gamma$ defines a pair of totally null planes intersecting on the null normal of the hypersurface in $T_\gamma {\mathcal H}$, one an $\alpha-$plane, one a $\beta-$plane. These plane fields can be integrable or contact, as explored in \cite{GG22a}.

There is a unique $\alpha-$surface in ${\mathbb L}$ containing $\gamma$ with tangent plane agreeing with the $\alpha-$plane at $\gamma$. Such a holomorphic Lagrangian surface is either the oriented lines through a point, or the oriented lines in a fixed direction. This is the neutral metric interpretation of the classical surface statement that a totally umbilic surface is either a sphere or a plane. 

Thus, the $\alpha-$plane at $\gamma\in{\mathbb L}$ identifies a point on each $\gamma\subset{\mathbb R}^3$ (albeit at infinity) which is the centre of the associated $\alpha-$surface. The locus of all these centres in ${\mathbb R}^3$ as one varies over ${\mathcal H}$ will be called the {\em focal set} of the null hypersurface. A null hypersurface is said to be {\em regular} if the focal set is a submanifold of ${\mathbb R}^3$.

\vspace{0.1in}
\begin{Prop}\label{p:nullity}
A regular null hypersurface ${\mathcal H}_n$ with focal set of dimension $n$ must be one of the following:
\begin{itemize}
\item[${\mathcal H}_0$:] The set of oriented lines parallel to a fixed plane,
\item[${\mathcal H}_1$:] The set of oriented lines through a fixed curve,
\item[${\mathcal H}_2$:] The set of oriented lines tangent to a fixed surface.
\end{itemize}
\end{Prop}
\vspace{0.1in}
Assuming the fixed curve and fixed surface are convex, we have ${\mathcal H}_0={\mathcal H}_2=S^1\times {\mathbb R}^2$ and ${\mathcal H}_1=S^2\times {\mathbb R}$. The null cone of a point $\gamma\in{\mathbb L}$ is clearly an example of null hypersurface ${\mathcal H}_1$, the fixed curve being the line $\gamma\subset{\mathbb R}^3$.

On the other hand, the formula presented in Section \ref{s:3.3} assumes data on a null hypersurface ${\mathcal H}_0$. Both the $\alpha-$ and $\beta-$planes in ${\mathcal H}_0$ are integrable, so it can be foliated by $\alpha-$surfaces (all the oriented lines in a fixed direction) and by $\beta-$surfaces (all oriented lines contained in a plane parallel to the fixed plane). 

The $\alpha$-foliation underpins the projection operator in the formula and it is not clear how the formula would look for data on null hypersurfaces of type ${\mathcal H}_1$ or ${\mathcal H}_2$, as the $\alpha-$planes are not in general integrable.

\begin{center}
    \includegraphics[scale=0.90]{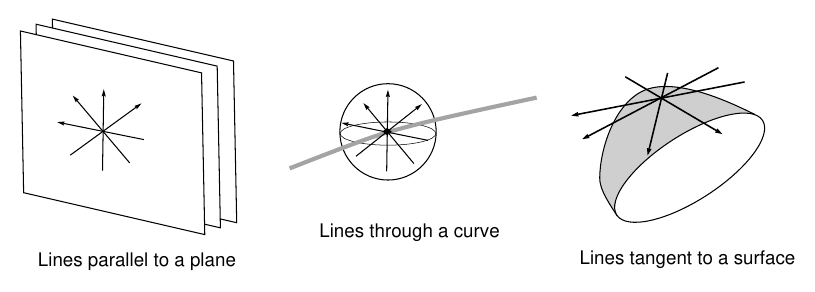}
\captionof{figure}{Regular null hypersurfaces in oriented line space}\label{f:null}
\end{center}

In Figure \ref{f:null} the three types of regular null hypersurfaces ${\mathcal H}_0, {\mathcal H}_1, {\mathcal H}_2$ are shown. The left null hypersurface is ${\mathcal H}_0$,  the standard configuration for acquiring data in CT scans, and is discussed in Section \ref{s:3.3}. 

Reconstruction using either of the other two null hypersurfaces would have advantages if one seeks to reduce the amount of radiation exposure during the scan. In particular, using the oriented lines ${\mathcal H}_1$ through a fixed line would reduce the exposure of each point to a semi-circle of radiation rather than the full circle in the ${\mathcal H}_0$. On the other hand, using the oriented lines ${\mathcal H}_2$  tangent to a convex surface would leave the interior occluded, and hence shielded completely from radiation. Whether either of these two configurations can be practically acquired by a physical scanner is another matter.

\vspace{0.1in}

\subsection{Paracomplex Structures}\label{s:2.2}

The complex structure $J_0$ on the space of oriented geodesics of a 3-dimensional space form evaluated at an oriented geodesic is obtained by rotation through 90$^o$ about the geodesic. This almost complex structure is integrable in the sense of Nijinhuis, which for any almost complex structure $J$ says
\[
N_{ij}^k=J^m_j\partial_mJ^k_i-J^m_i\partial_mJ^k_j+J_m^k(\partial_iJ_j^m-\partial_jJ^m_i)=0,
\]
and thus a complex structure. This is due to the fact that the ambient space has constant curvature \cite{Hitchin82}. 

One can also take reflection of an oriented line in a fixed oriented line $\gamma\in{\mathbb L}$ to generate a map $J_1:T_\gamma {\mathbb L}\rightarrow T_\gamma {\mathbb L}$ such that $J_1^2=1$ and the $\pm1-$eigenspaces are 2-dimensional. This {\em almost paracomplex structure} is not integrable in the sense of Nijinhuis and thus not a {\em paracomplex structure}. It is however anti-isometric with respect to the canonical neutral metric ${\mathbb G}$:
\[
{\mathbb G}(J_1\cdot,J_1\cdot)=-{\mathbb G}(\cdot,\cdot).
\]

\vspace{0.1in}
\begin{Thm}\cite{GG22b}\label{t:1}
The space of oriented lines of Euclidean 3-space admits an invariant commuting triple $(J_0,J_1,J_2)$ of a complex structure, an almost paracomplex structure and an almost complex structure, respectively, satisfying $J_2=J_0J_1$.
The complex structure $J_0$ is isometric, while $J_1$ and $J_2$ are anti-isometric. Only $J_0$ is parallel w.r.t. ${\mathbb G}$, and only $J_0$ is integrable. 

Composing the neutral metric ${\mathbb G}$ with the (para)complex structures $J_0,J_1,J_2$ yields closed 2-forms $\Omega_0$ and $\Omega_1$, and a conformally flat, scalar flat, neutral metric  $\tilde{\mathbb G}$, respectively. The neutral 4-manifolds $({\mathbb L},{\mathbb G})$ and $({\mathbb L},\tilde{\mathbb G})$ are isometric. Only $J_0$ is parallel w.r.t. $\tilde{\mathbb G}$.
\end{Thm}
\vspace{0.1in}

An almost paracomplex structure is an example of an {\em almost product structure}, in which a splitting of the tangent space at each point of the manifold is given, in this case $4=2+2$. Such pointwise splittings can only be extended over a manifold subject to certain geometric and topological conditions. For example

\vspace{0.1in}
\begin{Thm}\cite{GG22b}\label{t:2}
A  conformally flat neutral metric on a 4-manifold that admits a parallel anti-isometric or isometric almost paracomplex structure has zero scalar curvature.
\end{Thm}
\vspace{0.1in}

The parallel condition for an isometric almost paracomplex structure can be expressed in terms of the first order invariants of the eigenplane distributions:
\vspace{0.1in}
\begin{Thm}\cite{GG22b}\label{t:3}
Let $j$ be an isometric almost paracomplex structure on a pseudo-Riemannian 4-manifold. Then $j$ is parallel iff the eigenplane distributions are tangent to a pair of mutually orthogonal foliations by totally geodesic surfaces.
\end{Thm}
\vspace{0.1in}

Canonical examples for neutral conformally flat metrics are the indefinite product of two surfaces of equal constant Gauss curvature, which have exactly this double foliation. It is instructive in this case to use the isometric paracomplex structure $j=I\oplus-I$ to flip the sign of the product metric. The result is a Riemannian metric which turns out to be Einstein. This construction holds more generally:

\vspace{0.1in}
\begin{Thm}\cite{GG22b}\label{t:4}
Let $(M,g)$ be a Riemannian $4$-manifold endowed with a parallel isometric paracomplex structure $j$, and let the associated neutral metric be $g'(\cdot,\cdot)=g(j\cdot,\cdot)$. Then, $g'$ is locally conformally flat if and only if $g$ is Einstein.
\end{Thm}
\vspace{0.1in}
This transformation will be used in Section \ref{s:4.3} to find global topological obstructions to parallel isometric paracomplex structures.

\vspace{0.1in}
\subsection{The Space of Oriented Geodesics of Hyperbolic 3-Space}\label{s:2.3}

In this section we consider the space ${\mathbb L}({\mathbb H}^3)$ of oriented geodesics in  three dimensional hyperbolic space ${\mathbb H}^3$ of constant sectional curvature $-1$. The canonical neutral metric on this space has been considered in detail \cite{GG10a} \cite{GG10b} \cite{Salvai07}, but its relation to the ultrahyperbolic equation has not.  To illustrate the ideas of this paper, and explore the commonality with the flat case, proofs are provided in this section. 

The space ${\mathbb L}({\mathbb H}^3)$ of oriented geodesics in hyperbolic 3-space is diffeomorphic to that of oriented lines ${\mathbb L}({\mathbb R}^3)$ in Euclidean 3-space ${\mathbb L}({\mathbb H}^3)={\mathbb L}({\mathbb R}^3)=T{\mathbb S}^2$, but the projection map does not have the same geometric significance. In fact each oriented geodesic has \textit{two} Gauss maps (the beginning and end directions at the boundary of the ball model for ${\mathbb H}^3$) and there is a natural embedding into $S^2\times S^2$. Thus it is natural to view  ${\mathbb L}({\mathbb H}^3)$ as $S^2\times S^2$ with the diagonal removed or, more geometrically, the {\it reflected} diagonal removed \cite{GG10b}.

The canonical neutral metric $\tilde{\mathbb G}$ on ${\mathbb L}({\mathbb H}^3)$ is conformally flat and scalar flat, thus relating the solutions of the flat ultrahyperbolic equation with harmonic functions, as in the case of ${\mathbb L}({\mathbb R}^3)$. 

\vspace{0.1in}
\begin{Thm}\label{p:lh3}
For any compactly supported or asymptotically constant function $f$ on hyperbolic 3-space, its X-ray transform is harmonic with respect to the canonical neutral metric:
\[
\triangle_{\tilde{\mathbb G}}u_f=0,
\]
where $\triangle_{\tilde{\mathbb G}}$ is the Laplacian of $\tilde{\mathbb G}$.
\end{Thm}
\begin{proof}
Consider the upper half-space model of hyperbolic 3-space ${\mathbb H}^3$, that is $(x_1,x_2,x_3)\in{\mathbb R}^3, x_3\in{\mathbb R}_{>0}$ with metric
\[
ds^2=\frac{dx_1^2+dx_2^2+dx_3^2}{x
_3^2}.
\]

We can locally model the space of oriented geodesics in this model by $(\xi,\eta)\in{\mathbb C}^2$ where the unit parameterised geodesic is \cite{GG10b}
\begin{equation}\label{e:uhsp}
z=x_1+ix_2=\eta+\frac{\tanh r}{\bar{\xi}} \qquad\qquad x_3=\frac{1}{|\xi|\cosh r}.
\end{equation}
With respect to these coordinates the neutral metric is
\[
d{s}^2=-\frac{i}{4}\left(\frac{1}{\xi^2}d\xi^2-\frac{1}{\bar{\xi}^2}d\bar{\xi}^2+\bar{\xi}^2d\eta^2-\xi^2d\bar{\eta}^2\right),
\]
and the Laplacian is 
\[
\triangle_{\tilde{\mathbb G}}u=8{\mbox Im }\left(\frac{1}{\bar{\xi}^2}\partial^2_\eta u+\partial_\xi(\xi^2\partial_\xi u)\right).
\]
Note that
\[
\frac{\partial}{\partial r}=\frac{1}{\cosh^2r}\left(\frac{1}{\bar{\xi}}\frac{\partial}{\partial z}+\frac{1}{{\xi}}\frac{\partial}{\partial \bar{z}}-\frac{\sinh r}{|{\xi}|}\frac{\partial}{\partial t}\right).
\]
Now a straight-forward calculation establishes the following identity
\[
\triangle_{\tilde{\mathbb G}}u_f=4i\int_{-\infty}^\infty \frac{\partial }{\partial r}\left(\frac{1}{\bar{\xi}}\partial_z f-\frac{1}{{\xi}}\partial_{\bar{z}} f\right)dr=4i\left[\frac{1}{\bar{\xi}}\partial_z f-\frac{1}{{\xi}}\partial_{\bar{z}} f\right]_{-\infty}^\infty .
\]
Thus, by integration by parts, as long as the transverse gradient of $f$ falls off at the boundary faster than $|\xi|$, the boundary terms vanish and we get 
\[
\triangle_{\tilde{\mathbb G}}u_f=0.
\]
\end{proof}
\vspace{0.1in}

In Section \ref{s:3.1} unit (pseudo-)circles in flat planes are proven to be the domains of integration of a mean value theorem for solutions of the ultrahyperbolic equation and to generate doubly ruled surfaces in the underlying ${\mathbb R}^3$. We now present a local conformally flat coordinate system for ${\mathbb L}({\mathbb H}^3)$ using the hyperboloid model of hyperbolic 3-space ${\mathbb H}^3$, which lets one explicitly construct such doubly ruled surfaces in ${\mathbb H}^3$.

In the hyperboloid model in Minkowski space ${\mathbb R}^{3+1}$, ${\mathbb H}^3$ is the hyperboloid $x_0^2-x_1^2-x_2^2-x_3^2=1$ and the oriented geodesics are the intersections with oriented planes of Lorentz signature through the origin in ${\mathbb R}^{3+1}$. 

An oriented geodesic in ${\mathbb H}^3$ in the ball model can be uniquely determined by the directions at the boundary $(\mu_1,\mu_2)\in{\mathbb S}^2\times{\mathbb S}^2$. 
These directions $(\mu_1,\mu_2)$ are exactly the null directions on the Lorentz plane. 

The relationships between the complex coordinates $(\mu_1,\mu_2)\in{\mathbb C}^2$ obtained by stereographic projection on each ${\mathbb S}^2$ factor and the complex coordinates $(\xi,\eta)$ introduced in Theorem \ref{p:lh3} is
\[
\xi={\textstyle{\frac{1}{2}}}\left(\bar{\mu}_1+{\textstyle{\frac{1}{{\mu}_2}}}\right)^{-1}\qquad\qquad \eta={\textstyle{\frac{1}{2}}}\left(-{\mu}_1+{\textstyle{\frac{1}{\bar{\mu}_2}}}\right).
\]

\vspace{0.1in}
\begin{Prop}\label{p:concoo}
If $(\mu_1,\mu_2)$ are the standard holomorphic coordinates on ${\mathbb L}({\mathbb H}^3)$, consider the complex combination
\[
Z_1=\frac{(1+\mu_2\bar{\mu}_2)\bar{\mu}_1+(1+\mu_1\bar{\mu}_1)\bar{\mu}_2+i[(1-\mu_2\bar{\mu}_2)\bar{\mu}_1-(1-\mu_1\bar{\mu}_1)\bar{\mu}_2]}{1-\mu_1\bar{\mu}_1\mu_2\bar{\mu}_2}
\]
\[
Z_2=\frac{(1+\mu_2\bar{\mu}_2)\bar{\mu}_1+(1+\mu_1\bar{\mu}_1)\bar{\mu}_2-i[(1-\mu_2\bar{\mu}_2)\bar{\mu}_1-(1-\mu_1\bar{\mu}_1)\bar{\mu}_2]}{1-\mu_1\bar{\mu}_1\mu_2\bar{\mu}_2}.
\]
The flat neutral metric $ds^2=dZ_1d\bar{Z}_1-dZ_2d\bar{Z}_2$ pulled back by the above is equal to $\Omega^2 \tilde{\mathbb G}$ where
\[
\Omega=\frac{|1+\mu_1\bar{\mu}_2|^2}{1-|\mu_1|^2|\mu_2|^2}.
\]
The inverse mapping from $(\mu_1,\mu_2)$ to $(Z_1,Z_2)$ is given by
\begin{equation}\label{e:conflh3a}
\mu_1={\textstyle{\frac{1}{2}}}(\bar{A}+\bar{B})-\frac{\bar{A}-\bar{B}}{2|A-B|^2}\left(|A|^2-|B|^2+2- \sqrt{(|A|^2-|B|^2+2)^2-|A-B|^2|A+B|^2}\right)
\end{equation}
\begin{equation}\label{e:conflh3b}
\mu_2={\textstyle{\frac{1}{2}}}(\bar{A}-\bar{B})-\frac{(\bar{A}+\bar{B})}{2|A+B|^2}\left(|A|^2-|B|^2+2- \sqrt{(|A|^2-|B|^2+2)^2-|A-B|^2|A+B|^2}\right)
\end{equation}
where $A={\textstyle{\frac{1}{2}}}(Z_1+Z_2)$ and $B={\textstyle{\frac{1}{2i}}}(Z_1-Z_2)$.
\end{Prop}
\begin{proof}
A direct calculation.
\end{proof}
In Section \ref{s:3.2} these transformations will be used to construct surfaces in ${\mathbb H}^3$ that are ruled by geodesics in two distinct ways - doubly ruled surfaces.

\vspace{0.1in}
\section{The Ultrahyperbolic Equation}\label{s:3}

In this section solutions of the ultrahyperbolic equation (\ref{e:uhe}) are studied. A mean value property for such solutions is presented along with its interpretation in terms of doubly ruled surfaces in ${\mathbb R}^3$. Classically it was known that a non-flat doubly ruled surface in ${\mathbb R}^3$ is either a one-sheeted hyperboloid or a hyperbolic paraboloid \cite{hcv}. The construction of doubly ruled surfaces is extended to hyperbolic 3-space and the analogue of the 1-sheeted hyperboloid is exhibited. An explicit geometric formula is then given for the ultrahyperbolic equation with data given on a certain null hypersurface.

\vspace{0.1in}
\subsection{Mean Value Theorem}\label{s:3.1}

The X-ray transform takes a function $f:{\mathbb  R}^3\rightarrow{\mathbb R}$ to $u_f:{\mathbb L}\rightarrow{\mathbb R}$ by integrating over lines. In 1937 Fritz John showed that if a function $f$ satisfies certain fall-off conditions at infinity (which hold for compactly supported functions), then $u_f$ satisfies the ultrahyperbolic equation (\ref{e:uhe}), \cite{fjohn}. 

The link between the ultrahyperbolic equation (\ref{e:uhe}) and the neutral metric is

\vspace{0.1in}
\begin{Thm}\cite{CG22a}
Let $u:{\mathbb R}^{2,2}\rightarrow{\mathbb R}$ and $v:{\mathbb L}\rightarrow{\mathbb R}$ be related by $v=\Omega^{-1} u$, where $\Omega$ is the conformal factor.

Then $u$ is a solution of the ultrahyperbolic equation (\ref{e:uhe}) iff $v$ is in the kernel of the Laplacian of the neutral metric: $\Delta_{\mathbb G} v=0$.
\end{Thm}
\vspace{0.1in}

Leifur Asgeirsson \cite{LA} had earlier shown that solutions of the ultrahyperbolic equation satisfy a mean value property. In particular, for $u:{\mathbb R}^{2,2}\rightarrow{\mathbb R}$ a solution of equation (\ref{e:uhe}) satisfies

\begin{equation} \label{asgeirsson}
\int_0^{2\pi} u(a+r\cos\theta, b+r\sin\theta, c, d) \ d\theta
= 
\int_0^{2\pi} u(a,b,c+r\cos\theta, d+r\sin\theta) \ d\theta,
\end{equation}
for all $a,b,c,d\in\mathbb R$ and $r>0$. The two domains of integration are circles of equal radius lying in a pair of orthogonal planes $\pi,\pi^\perp$ in ${\mathbb R}^{2,2}$ with definite induced metrics on them. 

It can be shown that the mean value theorem holds over a much larger class of curves, namely the image of these circles under any conformal map of  ${\mathbb R}^{2,2}$. We refer to such curves as {\em conjugate conics} and these turn out to be pairs of circles, hyperbolae and parabolae lying in orthogonal planes of various signatures:

\vspace{0.1in}
\begin{Thm} \cite{CG22a} \cite{CG22b}
Let $S$ and $S^\perp$ be curves contained in orthogonal affine planes $\pi$ and $\pi^\perp$ in ${\mathbb R}^{2,2}$, respectively, which are one of the following pairs:
\begin{enumerate}
\item \emph{Circles} with equal and opposite radii $\pm r_0$ when the two planes are definite,
\item \emph{Hyperbolae} with equal and opposite radii $\pm r_0$ when the two planes are indefinite,
\item \emph{Parabolae} in non-intersecting degenerate affine planes determined by the property that every point on $S\subset\pi$ is null separated from every point on $S^\perp\subset\pi^\perp$. 
\end{enumerate}
Then the following mean value property holds for any solution $u$ of the ultrahyperbolic equation:
\[
\int_S u \ dl = \int_{S^\perp} u \ dl,
\]
where $dl$ is the line element induced on the curves  by the flat metric $g$.
\end{Thm}
\vspace{0.1in}
One can view this as a conformal extension of the original mean value theorem, one that intertwines the classical conic sections, the ultrahyperbolic equation and neutral geometry.

\vspace{0.1in}
\subsection{Doubly Ruled Surfaces}\label{s:3.2}

John also pointed out the relationship between the two circles in Asgeirsson's theorem and the double ruling of the hyperboloid of 1 sheet \cite{fjohn}. In fact, conjugate conics have been shown to correspond to the pairs of families of lines of all non-planar doubly ruled surfaces in ${\mathbb R}^3$.

\vspace{0.1in}
\begin{Thm} \cite{CG22b}\label{rulsurf}
Let $S,S^\perp$ be two curves in $\mathbb R^{2,2}$ representing the two one-parameter families of lines $L,L^\perp$ in ${\mathbb R}^3$. Then $S,S^\perp$ are a pair of conjugate conics in $\mathbb R^{2,2}$ if and only if $L$ and $L^\perp$ are the two families of generating lines of a non-planar doubly ruled surface in $3$-space.
\end{Thm}
\vspace{0.1in}

The geometric reason these curves yield a doubly ruled surface is that every point on one curve is zero distance from every point on the other curve - this follows from the neutral Pythagoras Theorem! But, as mentioned earlier, zero distance between oriented lines implies intersection, we see that every line of one ruled surface intersects every line of the other ruling, hence the double ruling. 

While this result was originally proven in ${\mathbb R}^3$, it holds in any 3-dimensional space of constant curvature, where the canonical neutral K\"ahler metric plays the same role. To demonstrate this, let us construct doubly ruled surfaces in 3-dimensional hyperbolic space ${\mathbb H}^3$. 

Recall the conformal coordinates for ${\mathbb L}({\mathbb H}^3)$ given in equations (\ref{e:conflh3a}) and (\ref{e:conflh3b}). To generate the hyperbolic equivalent of the 1-sheeted hyperboloid, the two curves (parameterized by $u$) are circles of radii $\pm r_0$ in two definite planes:
\[
Z_1=r_0e^{iu}\qquad\qquad Z_2=0,
\]
and
\[
Z_1=0\qquad\qquad Z_2=r_0e^{iu}.
\]

For the curves we can view the doubly ruled surfaces in either the upper half-space model or the ball model of ${\mathbb H}^3$. For the former, one uses the equations (\ref{e:uhsp}), while for the latter one can use
\[
x_1+ix_2=\frac{\mu_2(1+\mu_1\bar{\mu}_1)e^v-\mu_1(1+\mu_2\bar{\mu}_2)e^{-v}}{(1+\mu_1\bar{\mu}_1)(1+\mu_2\bar{\mu}_2)\cosh v+\left[(1+\mu_1\bar{\mu}_2)(1+\mu_2\bar{\mu}_1)(1+\mu_1\bar{\mu}_1)(1+\mu_2\bar{\mu}_2)\right]^{\scriptstyle{\frac{1}{2}}}}
\]
\[
x_3=\frac{(1+\mu_1\bar{\mu}_1)(1-\mu_2\bar{\mu}_2)e^v-(1+\mu_2\bar{\mu}_2)(1-\mu_1\bar{\mu}_1)e^{-v}}{2\left((1+\mu_1\bar{\mu}_1)(1+\mu_2\bar{\mu}_2)\cosh v+\left[(1+\mu_1\bar{\mu}_2)(1+\mu_2\bar{\mu}_1)(1+\mu_1\bar{\mu}_1)(1+\mu_2\bar{\mu}_2)\right]^{\scriptstyle{\frac{1}{2}}}\right)}.
\]
Figure 1 is a plot of a doubly ruled surface in the upper half-space model while Figure 2 is in the ball model of hyperbolic 3-space. These are the hyperbolic equivalent of the 1-sheeted hyperboloid, although they satisfy a fourth order (rather than second order) polynomial equation.

\begin{figure}
    \centering
    \begin{minipage}{0.45\textwidth}
        \centering
        \includegraphics[width=1.2\textwidth]{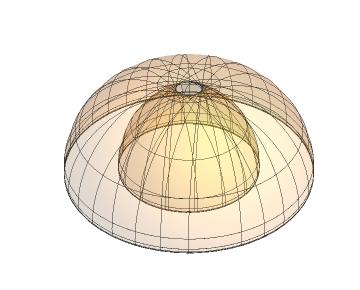} 
        \caption{Halfspace model}
    \end{minipage}\hfill
    \begin{minipage}{0.45\textwidth}
        \centering
        \includegraphics[width=0.9\textwidth]{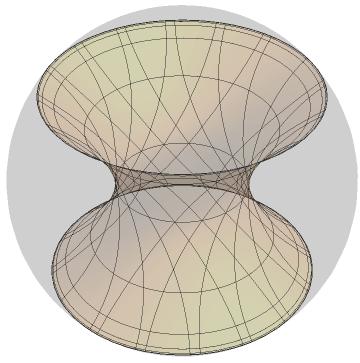} 
        \caption{Ball model}
    \end{minipage}
\end{figure}

\vspace{0.1in}
\subsection{Cauchy Problem for the Ultrahyperbolic Equation}\label{s:3.3}

One way to reconcile the difference between the dimension of ${\mathbb L}({\mathbb R}^3)$ and that of ${\mathbb R}^3$ is to consider the problem of determining the value of a solution $v:{\mathbb L}\rightarrow{\mathbb R}$ of the Laplace equation
\[
\triangle_{\mathbb G}v=0,
\]
on all of oriented line space ${\mathbb L}$, given only the values of the function on a null hypersurface ${\mathcal H}\subset {\mathbb L}$.

Consider the case where the data is known on the hypersurface generated by all oriented lines parallel to a fixed plane in $P_0\subset{\mathbb R}^3$ - the case of regular dimension zero focal set ${\mathcal H}_0$ in Proposition \ref{p:nullity}. 

This null hypersurface is suitable as a boundary for the Cauchy problem, as proven by John \cite{fjohn}. In fact, it can be foliated both by $\alpha-$planes and $\beta-$planes - the former being the oriented lines parallel to $P_0$ in a fixed direction, while the latter are all oriented lines parallel to $P_0$ at a fixed height.

Denote
\[
{\mathcal H}=\{\gamma\in{\mathbb L}\;|\;\;\gamma\parallel P_0\;\}.
\]
Clearly ${\mathcal H}={\mathbb S}^1\times{\mathbb C}$ and for convenience, suppose that $P_0$ is horizontal in standard coordinates, so that in complex coordinates the hypersurface is $\xi=e^{i\theta}$, since the only restriction on the oriented line is that its direction lies along the equator. 

The distance between parallel lines in ${\mathbb R}^3$ induces the metric (\ref{e:fibmet}) and associated distance function $\|.\|$. In fact, there is an invariant metric on ${\mathcal H}$ with volume form $d^3Vol=d\eta\;d\bar{\eta}\;d\theta$. 

Suppose that $\gamma_0\notin {\mathcal H}$ and consider the intersection of this null hypersurface with the null cone $C_0(\gamma_0)\cap{\mathcal H}={\mathbb S}^1\times{\mathbb R}$. This surface intersects each fibre in an affine line. Let $Pr_0(\gamma)$ be the projection of $\gamma$ onto this affine line with respect to the fibre metric: $Pr_0:{\mathbb S}^1\times{\mathbb R}^2\rightarrow{\mathbb S}^1\times{\mathbb R}$.

We now prove the following explicit geometric formula that determines the value of a solution of the ultrahyperbolic equation from its value on the null hypersurface of type ${\mathcal H}_0$ in ${\mathbb L}$:

\vspace{0.1in}
\begin{Thm}\label{t:uheinv}
If $v:{\mathbb L}\rightarrow{\mathbb R}$ is a function satisfying the ultrahyperbolic equation, then at an oriented line $\gamma_0$ 
\[
v(\gamma_0)=-{\textstyle{\frac{1}{2\pi^2}}}\iiint_{\gamma\in{\mathcal H}} \frac{v(\gamma)-v(Pr_0(\gamma))}{\|\gamma-Pr_0(\gamma)\|^2}\;d^3Vol,
\]
where $Pr_0(\gamma)$ is projection onto the intersection of the null cone of $\gamma_0$ with the $\alpha$-plane through $\gamma$ that lies in the null hypersurface ${\mathcal H}$.
\end{Thm}
\begin{proof}
Our starting point is Fritz John's formula (equation (13) of \cite{fjohn}) which gives the solution of the ultrahyperbolic equation at an oriented line $\gamma_0$ by the cylindrical average over all planes parallel to $\gamma_0$:

\begin{equation}\label{e:jinv}
v(\gamma_0)=-{\textstyle{\frac{1}{\pi}}}\int_0^\infty \frac{F(R)-F(0)}{R^2}dR,
\end{equation}
where
\[
F(R)={\textstyle{\frac{1}{2\pi}}}\int_0^{2\pi} \iint_{P_{(R,\alpha)}}\rho(r,s)drds\;d\alpha,
\]
$P_{(R,\alpha)}$ is the plane parallel to $\gamma_0$ at a distance $R$ and angle $\alpha$, and $(r,s)$ are flat coordinates on that plane.

Consider the map
\begin{equation}\label{e:z1}
z=\frac{1}{1+\nu\bar{\nu}}\left(2\nu R+(e^{iA}-\nu^2e^{-iA})r+i(e^{iA}+\nu^2e^{-iA})s\right)
\end{equation}
\begin{equation}\label{e:t}
x_3=\frac{1}{1+\nu\bar{\nu}}\left((1-\nu\bar{\nu})R-(\bar{\nu}e^{iA}+\nu e^{-iA})r -i(\bar{\nu}e^{iA}-\nu e^{-iA})s\right).
\end{equation}

For fixed $R\in{\mathbb R}$, $\nu\in{\mathbb C}$ and $A\in[0,2\pi)$, the map $(r,s)\mapsto(z(r,s),x_3(r,s))\in{\mathbb R}^3$ paramaterizes the plane a distance $R$ from the origin with normal direction $\nu$. Changing $A$ rotates the $r$- and $s$-axes in the plane. 

By a translation we can assume $\gamma_0$ contains the origin and so has complex coordinates $(\xi=\xi_0,\eta=0)$. Let us restrict attention to planes that are parallel $\gamma_0$. Thus the normal direction of $P_{(R,\nu)}$ is perpendicular to the direction of $\gamma_0$, we have
\[
\nu=\frac{\xi_0+e^{i\alpha}}{1-\bar{\xi}_0e^{i\alpha}},
\]
where $\alpha\in[0,2\pi)$.

The quantity $R$ is then just the distance from the plane to the line $\gamma_0$. Finally we want to rotate the ruling by $s$ on the plane so that it is horizontal and thus a curve in ${\mathcal H}$. Clearly this is achieved by
\[
\nu=r_0e^{iA},
\]
or more explicitly
\[
A={\textstyle{\frac{1}{2i}}}\ln \left[\frac{(\xi_0+e^{i\alpha})(1-{\xi}_0e^{-i\alpha})}{(\bar{\xi}_0+e^{-i\alpha})(1-\bar{\xi}_0e^{i\alpha})}\right]
\qquad
r_0=\left[\frac{(\xi_0+e^{i\alpha})(\bar{\xi}_0+e^{-i\alpha})}{(1-{\xi}_0e^{-i\alpha})(1-\bar{\xi}_0e^{i\alpha})}\right]^{\scriptstyle{\frac{1}{2}}}.
\]
The first of these is invertible for fixed $\xi_0$, $A \leftrightarrow \alpha$.

The horizontal ruling for $P_{(A,\alpha)}$ is
\[
z=\frac{2\nu}{1+\nu\bar{\nu}}R+\frac{1-\nu\bar{\nu}}{1+\nu\bar{\nu}}re^{iA}+ise^{iA}
\]
\[
x_3=\frac{1-\nu\bar{\nu}}{1+\nu\bar{\nu}}R-\frac{2|\nu|}{1+\nu\bar{\nu}}r.
\]
The direction of the ruling is
\[
\frac{\partial}{\partial s}=ie^{iA}\frac{\partial}{\partial z}-ie^{-iA}\frac{\partial}{\partial \bar{z}}
\]
so that the complex coordinates are $\xi=ie^{iA}$ and 
\[
\eta={\textstyle{\frac{1}{2}}}(z-2x_3\xi-\bar{z}\xi^2)=-(r-iR)\left(\frac{r_0-i}{r_0+i}\right)e^{iA}.
\]
Thus we have parameterized ${\mathcal H}$ by coordinates $(R,\alpha,r)$ and a straightforward calculation shows that the fibre metric is simply
\[
d\eta d\bar{\eta}=dR^2+dr^2 \qquad\qquad{\mbox{ and}} \qquad\qquad d^3Vol=drdRd\alpha.
\]
The null cone of $\gamma_0$ consists of all lines that either intersect or are parallel to it. For non-horizontal $\gamma_0$ the null cone intersects the null hypersurface ${\mathcal H}$ at the lines that intersect $\gamma_0$, namely those with coordinates $(R=0,\alpha,r)$ which is a line through the origin in each fibre. We have chosen $\gamma_0$ to contain the origin in ${\mathbb R}^3$, which is why the line in the fibre is through the origin. More generally the intersection of the null cone with a fibre is an affine line (not necessarily through the origin), as claimed.

Thus the fibre projection is simply $Pr_0(R,\alpha,r)=(0,\alpha,r)$ and
\[
R=\|\gamma-Pr_0(\gamma)\|.
\]
Now putting this together with the integral formula
\begin{align}
v(\gamma_0)=&-{\textstyle{\frac{1}{2\pi^2}}}\int_0^\infty \int_0^{2\pi} \frac{1}{R^2}\left[\iint_{P_{(R,\alpha)}}\rho(r,s)drds-\iint_{P_{(0,\alpha)}}\rho(r,s)drds\right]\;dRd\alpha \nonumber\\
=&-{\textstyle{\frac{1}{2\pi^2}}}\int_0^\infty \int_0^{2\pi} \int_{-\infty}^\infty \frac{v(R,\alpha,r)-v(0,\alpha,r)}{R^2}dr dR d\alpha\nonumber\\
=&-{\textstyle{\frac{1}{2\pi^2}}}\iiint_{\gamma\in{\mathcal H}} \frac{v(\gamma)-v(Pr_0(\gamma))}{\|\gamma-Pr_0(\gamma)\|^2}\;d^3Vol,\nonumber
\end{align}
as claimed.
\end{proof}

\vspace{0.1in}
\section{Topological Considerations}\label{s:4}
In this section global topological aspects of neutral metrics and almost product structures are explored. These include the relationship between umbilic points on surfaces in ${\mathbb R}^3$ and complex points on Lagrangian surfaces in ${\mathbb L}$, and an associated boundary value problem for the Cauchy-Riemann operator. The significance of these constructions for a number of conjectures from classical surface theory is indicated.

Some background on the problems of 4-manifold topology are discussed with particular attention to codimension two. The significance of neutral metrics to these issues is that they are uniquely capable of quantifying codimension two topological phenomena, and thus can be used as geometric tools to resolve certain long-standing questions. For the case of closed 4-manifolds, we end with a discussion of topological obstructions that arise to certain neutral geometric structures.

\vspace{0.1in}
\subsection{Global Results}\label{s:4.1}
Topological aspects of neutral metrics become evident in the identification of complex points on Lagrangian surfaces in ${\mathbb L}$ with umbilic points on surfaces in ${\mathbb R}^3$ \cite{GK04}. 

The Lagrangian surface $\Sigma\subset{|mathbb L}$ is formed by the oriented normal lines to the surface $S\subset{\mathbb R}^3$ and the index $i(p)\in{\mathbb Z}/2$ of an isolated umbilic point $p\in S$ on a convex surface is exactly one half of the complex index of the corresponding complex point $\gamma\in\Sigma$: $I(\gamma)=2i(p)\in{\mathbb Z}$. Thus problems of classical surface theory can be explored through studying Lagrangian surfaces in the four dimensional space of oriented lines ${\mathbb L}$ with its neutral metric ${\mathbb G}$.

The metric induced on a Lagrangian surface is Lorentz or degenerate - the degenerate points being the umbilic points of $S$ and the null cone at $\gamma$ being the principal directions of $S$ at $p$. The indices of isolated umbilic points carry geometric information from the neutral metric and vice versa.

If an isolated umbilic point $p$ has half-integer index then the principal foliation around $p$ is non-orientable - it defines a line field rather than a vector field about the umbilic point. The foliation is orientable if the index is an integer. The following theorem establishes a topological version of a result of Ferdinand Joachimsthal \cite{Joach46} for surfaces intersecting at a constant angle:

\vspace{0.1in}
\begin{Thm} \cite{GK19b}
If $S_1$ and $S_2$ are smooth convex surfaces intersecting with constant angle along a curve that is not umbilic in either $S_1$ or $S_2$, then the principal foliations of the two surfaces along the curve are either both orientable, or both non-orientable. 

That is, if $i_1\in{\mathbb Z}/2$ is the  sum of the umbilic indices inside the curve of intersection on $S_1$ and $i_2\in{\mathbb Z}/2$ is the  sum of the umbilic indices inside the curve of intersection on $S_2$ then
\[
2i_1=2i_2\;{\mbox{\rm  mod }}2.
\]
\end{Thm}
\vspace{0.1in}

Pushing deeper, if one considers the problem of finding a holomorphic disc in ${\mathbb L}$ whose boundary lies on a given Lagrangian surface $\Sigma$, one encounters a classical problem of Riemann-Hilbert for the Cauchy-Riemann operator. Given a totally real surface $\Sigma$ in a complex surface ${\mathbb M}$, the Riemann-Hilbert problem seeks a map $f:(D,\partial D)\rightarrow({\mathbb M},\Sigma)$ which is holomorphic: it lies in the kernel of the Cauchy-Riemann operator $\bar{\partial}f=0$. For this to be an elliptic boundary value problem it is required that the boundary surface $\Sigma$ be totally real i.e. has no complex points. In the Riemannian case Lagrangian implies  totally real, and so Lagrangian boundary conditions are often used when the ambient metric is Riemannian.

In our case, due to the neutral signature of the metric formed by the composition of the symplectic structure (which defines Lagrangian) and the complex structure (which defines holomorphic), new features arise. In particular, Lagrangian surfaces may not be totally real, and therefore at complex points they are not suitable as a boundary condition for the $\bar{\partial}$-operator. If, however, the boundary surface is assumed to be spacelike with respect to the metric, then by the neutral Wirtinger identity it is also totally real and is suitable. 

The deformation from Lagrangian to spacelike by the addition of a holomorphic twist can be achieved over an open hemisphere. This \textit{contactification} of the problem throws away the surface $S$ in ${\mathbb R}^3$, as the perturbed spacelike surface $\tilde{\Sigma}$ in ${\mathbb L}({\mathbb R}^3)$ forms a 2-parameter family of twisting oriented lines in ${\mathbb R}^3$ that are not orthogonal to any surface. Any holomorphic disc with boundary lying on $\tilde{\Sigma}$ yields a holomorphic disc with boundary lying on $\Sigma$ by subtracting the holomorphic twist and so the problems are equivalent over a hemisphere.

The Riemann-Hilbert problem then follows the standard case, with the linearisation at a solution defining an elliptic boundary value problem with analytic index ${\mathcal I}$ given by
\[
{\mathcal I} = {\mbox{\rm Dim Ker }}\bar{\partial}-{\mbox{\rm Dim Coker }}\bar{\partial}.
\]
The analytic index for the problem is well-known to be related to the Keller-Maslov index $\mu(\partial D,\Sigma)$ along the boundary by
\[
{\mathcal I}=\mu+2.
\]
The Keller-Maslov index in the case of a section of ${\mathbb L}$ is given by the sum $i$ of the umbilic indices inside the curve $\partial D$ in the boundary $\Sigma$, as viewed in ${\mathbb R}^3$ \cite{GK04}:
\[
\mu=4i.
\]
For the Keller-Maslov class to control the dimension of the space of holomorphic discs, one needs the dimension of the cokernel to be zero. If the problem is Fredholm regular, by a small perturbation the cokernel vanishes and the space of holomorphic discs is indeed determined by the number of enclosed umbilic points.

Remarkably, the Riemann-Hilbert problem associated with a convex sphere containing a single umbilic point \textit{is} Fredholm regular: 
\vspace{0.1in}
\begin{Thm}\label{t:fred}\cite{GK20a}
Let $\Sigma\subset{\mathbb L}$ be a Lagrangian sphere with a single isolated complex point. Then the Riemann-Hilbert problem with boundary $\Sigma$ is Fredholm regular. 
\end{Thm}
\vspace{0.1in}
The reason behind this result is that the Euclidean isometry group acts holomorphically and symplectically on ${\mathbb L}$, thus preserving the problem. The action is also transitive and so fixing the single complex point one considers the equivariant problem, the result being that it is Fredholm regular, as in the totally real case. 

The non-existence of a convex sphere containing a single umbilic point is the famous conjecture of Constantin Carath\'eodory, and Theorem \ref{t:fred} gives the reason the Conjecture is true. Namely, were such a remarkable surface $S$ to exist, the Riemann-Hilbert problem with boundary given by the normal lines $\Sigma$ would be Fredholm regular and so have the property that the dimension of the space of parameterised holomorphic discs with boundary lying on it would be entirely determined by the number of umbilic points in the interior on $S$.
\begin{equation}\label{e:cc1}
{\mathcal I}={\mbox{\rm Dim Ker }}\bar{\partial}=4i+2
\end{equation}
This property would also hold for a dense set of perturbations of $S$ in an appropriate function space. To show that such a surface $S$ cannot exist, one can seek to find violations of equation (\ref{e:cc1}), in particular, a holomorphic disc which encloses a totally real disc on the boundary $\Sigma$. 

By equation (\ref{e:cc1}), if the boundary encloses a totally real disc, then ${\mathcal I}=2$. However, since the M\"obius group acts on the space of parameterized holomorphic discs, the space of unparameterized holomorphic discs is $2-3=-1$. Thus, over an umbilic-free region of the remarkable surface $S$ it should be impossible to solve the $\bar{\partial}$-problem. 

The proof of the Carath\'eodory Conjecture in \cite{GK11} follows from the existence of holomorphic discs with boundary enclosing umbilic-free regions, as established by evolving to them using mean curvature flow of a spacelike surface in ${\mathbb L}$, thus disproving equation (\ref{e:cc1}). 

At this point in time two thirds of the proof given in \cite{GK11} has appeared in print, with  the final part containing the boundary estimates for mean curvature flow currently under review. 

In fact, the interior estimates required to prove long time existence and convergence hold for more general spacelike mean curvature flow with respect to indefinite metrics satisfying certain curvature conditions \cite{GK19a}. 

The final step of the proof of the Conjecture is the establishment of boundary estimates for mean curvature flow in ${\mathbb L}$ and sufficient control to show that the flow weakly converges in an appropriate function space to a holomorphic disc. The boundary conditions used for mean curvature flow (a second order system) include a constant angle condition and an asymptotic holomorphicity condition. 

The constant angle condition is defined between a pair of spacelike planes that intersect along a line and is hyperbolic in nature. The asymptotic holomorphicity condition ensures that the ultimate disc is holomorphic rather than just maximal. 

The sizes of the constant hyperbolic angle and the added holomorphic twist are free parameters in the evolution and can be used to control the flowing surface. If one views it as a codimension two capillary problem, the effect of the parameter changes is to increase the friction at the boundary, stopping it from skating off the hemisphere, thus preserving strict parabolicity.

An analogous result in the rotationally symmetric case for mean curvature flow in the space of oriented lines with Dirichlet and Neumann boundary conditions shows that the evolving surface can be made to converge to a holomorphic disc  - in this case to a family of holomorphic discs called the Bishop family \cite{Bishop65} - or to a maximal surface, depending on the boundary condition imposed \cite{GK23a}.

For the full flow one can then show that:
\vspace{0.1in}
\begin{Thm}\label{t:ashol} \cite{GK11}
Let $S$ be a $C^{3+\alpha}$ smooth oriented convex surface in ${\mathbb R}^3$ without umbilic points and suppose that the Gauss image of $S$ contains a closed hemisphere. Let $\Sigma\subset {\mathbb L}$ be the oriented normal lines of $S$ forming a Lagrangian surface in the space of oriented lines. 

Then $\exists f:D\rightarrow {\mathbb L}$ with $f\in C^{1+\alpha}_{loc}(D)\cap C^0(\overline{D})$ satisfying
\begin{enumerate}
\item[(i)] $f$ is holomorphic,
\item[(ii)] $f(\partial D)\subset\Sigma$.
\end{enumerate}
\end{Thm} 
\vspace{0.1in}
This would conclude the proof of the Carath\'eodory Conjecture for $C^{3+\alpha}$ smooth surfaces.

The appearance of Gauss hemispheres here is noteworthy, for this meets with a conjecture of Victor Toponogov that a complete convex plane must have an umbilic point, albeit at infinity \cite{Top}. Toponogov showed that such planes have hemispheres as Gauss image and established his conjecture under certain fall-off conditions at infinity.

In fact, the same reasoning as above that pits Fredholm regularity against mean curvature flow proves the Toponogov Conjecture:
\vspace{0.1in}
\begin{Thm} \cite{GK20b}
Every $C^{3+\alpha}$-smooth complete convex embedding of the plane $P$, satisfies $\inf_P |\kappa_1-\kappa_2|=0$.
\end{Thm}
\vspace{0.1in}
The proof follows from applying Theorem \ref{t:ashol} in this case, while Fredholm regularity is established easily, as a putative counter-example is by assumption totally real (even at infinity).

Without the high degree of symmetry of the Euclidean group, one would not expect Fredholm regularity to hold and this obstructs the generalisation of the Carath\'eodory Conjecture to non-Euclidean ambient metrics. This turns out to be the case and the delicate nature of the problem is revealed:

\vspace{0.1in}
\begin{Thm}\cite{G20}
For all $\epsilon>0$, there exists a smooth Riemannian metric $g$ on ${\mathbb R}^3$ and a smooth strictly convex 2-sphere $S\subset{\mathbb R}^3$ such that
\begin{itemize}
    \item[(i)] $S$ has a single umbilic point,
    \item[(ii)] $\|g-g_{0}\|^2\leq\epsilon$,
\end{itemize}
where $\|.\|$ is the $L_2$ norm on ${\mathbb R}^3$ with respect to the flat metric $g_0$.
\end{Thm}
\vspace{0.1in}
The proof here is constructive: the Euclidean metric is deformed while keeping the standard round 2-sphere fixed (although not round in the deformed metric) and one can essentially brush the principal foliation of the surface into any configuration one chooses by changing the ambient geometry.  

Finally, establishing the local index bound $i(p)\leq1$ for any isolated umbilic point $p$ has long been the preferred route to proving the Carath\'eodory Conjecture in the real analytic case \cite{Ham41} \cite{Ivan02}.  The above methods can also be used to find a slightly weaker local index bound for isolated umbilics on smooth surfaces:

\vspace{0.1in}
\begin{Thm}\cite{GK12}
The index of an isolated umbilic $p$ on a $C^{3,\alpha}$ surface in ${\mathbb R}^3$ satisfies $i(p)<2$.
\end{Thm}
\vspace{0.1in}
The proof follows from the extension of Theorem \ref{t:fred} to surfaces of higher genus by removing hyperbolic umbilic points and adding totally real cross-caps to the Lagrangian section. The existence of holomorphic discs over open hemispheres again contradicts Fredholm regularity and the local index bound follows.

Once again, the role of the Euclidean isometry group is paramount, and even a small perturbation of the ambient metric means that the index bound does not hold.

\vspace{0.1in}
\begin{Thm}\cite{G20}
For all $\epsilon>0$ and $k\in{\mathbb Z}/2$, there exists a smooth Riemannian metric $g$ on ${\mathbb R}^3$ and a smooth embedded surface $S\subset{\mathbb R}^3$ such that
\begin{itemize}
    \item[(i)] $S$ has an isolated umbilic point of index $k$,
    \item[(ii)] $\|g-g_{0}\|^2\leq\epsilon$,
\end{itemize}
where $\|.\|$ is the $L_2$ norm on ${\mathbb R}^3$ with respect to the flat metric $g_0$.
\end{Thm}
\vspace{0.1in}
Finally, the local umbilic index bound $i(p)\leq1$ of Hamburger \cite{Ham41} for real analytic surfaces has recently been used to prove results on the zeros of certain holomorphic polynomials. In particular, a polynomial whose zero set is invariant under inversion in the unit circle is called {\em self-inversive} \cite{BaM52} \cite{JaS18} \cite{OhaR74} \cite{TSS} \cite{V17}. 

\vspace{0.1in}

\begin{Thm}\cite{GK23b}
Let $P_N$ be a polynomial of degree $N$ with self-inversive second derivative and suppose that none of the roots of $P_N$ lies on the unit circle. Then the number of roots (counted with multiplicity) of $P_N$ inside the unit circle is less than or equal to $\lfloor N/2\rfloor+1$.
\end{Thm}
\vspace{0.1in}
This result is in the spirit of a converse to the Gauss-Lucas theorem \cite{Marden} in which the zeros of the first derivative of a polynomial are restricted by the zeros of the polynomial. Here, however, by methods of differential geometry, the locations of the zeros of the second derivative restrict the zeros of the polynomial - the first such application. It is also worth noting that the result is sharp.

The method of proof is to take a polynomial with self-inversive second derivative and to construct a real analytic strictly convex with an isolated umbilic point whose index is determined by the number of zeros inside the unit circle.

\vspace{0.1in}
\subsection{Four Manifold Topology}\label{s:4.2}

The proof by Grigori Perelman of Thurston's Geometrization Conjecture \cite{Perelman02}\cite{Perelman03a}\cite{Perelman03b} naturally raises the question as to whether closed 4-dimensional manifolds can be geometrized in some way. The approach in three dimensions, however, does not apply in higher dimensions and even basic things are harder. 

For example, any finitely presented group can be the fundamental group of a smooth closed 4-manifold, while the fundamental group of a prime 3-manifold must be a quotient of the isometry group of one of the eight Thurston homogenous geometries \cite{Thurston79}, and so it is clear that new geometric paradigms are required. 

To make matters worse, while in three dimensions there is no distinction between smooth, piecewise-linear and topological structures on closed manifolds, in higher dimension this may not be true. If one considers open manifolds, these problems are compounded further. In each dimension $n\geq3$ there are uncountably many {\it fake} ${\mathbb R}^n$'s - open topological manifolds that are homotopy equivalent to, but not homeomorphic to ${\mathbb R}^n$ \cite{Curtis65}\cite{Glaser67}\cite{McMillan62}. While many of these involve infinite constructions, an example of Barry Mazur in dimension four requires only the attachment of two thickened cells \cite{Mazur61}.

Four dimensions also has its share of peculiar problems that do not arise in higher dimensions. In particular, the Whitney trick, in which closed loops are contracted to a point across a given disc, plays a major role in many higher dimensional results, for example Stephen Smale's proof of the h-cobordism theorem \cite{Smale62}. The issue is that, while in dimensions five and greater a generic 2-disc is embedded, in dimension four a generic 2-disc is only immersed and will have self-intersections, making it unsuitable to contract loops across.

Against this array of formidable difficulties, the Disc Theorem of Micheal Freedman \cite{Freedman82} utilizes a doubly infinite codimension two construction to claim that there is a topological work-around for the Whitney trick. This result leads to the proof of the topological Poincar\'e Conjecture in dimension four, as well as the classification of all simply connected closed topological 4-manifolds based almost entirely on their intersection form in the second homology. 

Contradictions with Donaldson's ground-breaking work on smooth 4-manifolds \cite{Donaldson83} lead to extraordinary families of exotic manifolds (homeomorphic but not diffeomorphic) not seen in any other dimension. Since the work of John Milnor \cite{Milnor59} it has been known that exotic differentiable structures in dimensions seven and above exist, but only in finitely many families. According to the Disc Theorem exotic differentiable structures in dimension four occur in uncountable families - indeed, no 4-manifold is known to have only countably many distinct differentiable structures.

Both the original Disk Theorem \cite{Freedman82} and subsequent attempts to complete it \cite{Behrens21} \cite{Freedman83} \cite{FQ14}  \cite{G19} \cite{Hartnett21} have depended upon the iterative attachment of 1- and 2-handles or there generalizations, as one attempts to push unwanted codimension two intersections to infinity. The ultimate homeomorphism that is sought is shown to exist using Bing shrinking and what is called decomposition space theory \cite{Bing58}.

One key aspect of these efforts is that they all involve codimension two constructions - gluing in thickened 2-discs or more general surfaces into 4-manifolds. The work in this survey involves geometric paradigms associated with neutral metrics which can gain more control of these codimension two constructions. 

Unlike Riemannian metrics which exist on all smooth manifolds, neutral metrics see the topology of the underlying manifold and can be used to express topological invariants. The next section considers closed 4-manifolds and illustrates the manner in which the existence of certain neutral metrics restricts the topology of the underlying 4-manifold. These are modest steps in the direction of understanding a tiny part of the wild world of 4-manifolds in which there is a splitting $4=2+2$.

\vspace{0.1in}
\subsection{Closed Neutral 4-manifolds}\label{s:4.3}

The simplest topological invariant of a closed 4-manifold $M$ is its {\em Euler number} $\chi(M)$. Let $H_n(M,{\mathbb R})$ be the $n^{th}$ homology group of $M$ with real coefficients and $b_n$ be the associated Betti numbers $n=0,1,...,4$. For a closed connected 4-manifold we have $b_0=b_4=1$, and $b_3=b_1$ by Poincar\'e duality and the Euler number is defined
\[
\chi(M)=\sum_{n=0}^4(-1)^n{\mbox{ dim}}\;H_n(M,{\mathbb R})=2-2b_1+b_2,
\]
The Chern-Gauss-Bonnet Theorem states that one can express this geometrically as
\[
\chi(M)={{\frac{\epsilon}{32\pi^2}}}\int_M|W(g)|^2-2|Ric(g)|^2+{\textstyle{\frac{2}{3}}}S^2\;d^4V_g,
\]
for any metric $g$ of definite ($\epsilon=1$) or neutral signature ($\epsilon=-1$) \cite{Law}. 

On a closed 4-manifold there is a natural symmetric bilinear pairing on the integral second homology $H_2(M,{\mathbb Z})$. It is the sum of the number of transverse intersection points between two surfaces representing the homology classes. 

The intersection form can be diagonalised over ${\mathbb R}$ and the number of positive and negative eigenvalues is denoted $b_+$ and $b_-$, respectively. Thus $b_2=b_++b_-$ and the {\em signature} $\tau(M)=b_+-b_-$ is another topological invariant of $M$.

The existence of a neutral metric on a closed 4-manifold is equivalent to the existence of a field of oriented tangent 2-planes on the manifold \cite{Mats91}. Moreover:

\vspace{0.1in}
\begin{Thm} \cite{HaH} \cite{kam02} \cite{Mats91}
Let $M$ be a closed 4-manifold admitting a neutral metric. Then
\begin{equation}\label{e:mats}
\chi(M)+\tau(M)=0{\mbox{ mod }}4
\qquad{\mbox and }\qquad
\chi(M)-\tau(M)=0{\mbox{ mod }}4.
\end{equation}

If $M$ is simply connected, these conditions are sufficient for the existence of a neutral metric.
\end{Thm}
\vspace{0.1in}

Thus, neither ${\mathbb S}^4$ nor ${\mathbb C}P^2$ admit a neutral metric, while the K3 manifold does. 

Given a neutral metric $g'$ on $M$, the Euler number and signature can be expressed in terms of curvature invariants by
\[
\chi(M)={{\frac{-1}{32\pi^2}}}\int_M|W^+|^2+|W^-|^2-2|Ric|^2+{\textstyle{\frac{2}{3}}}S^2\;d^4V_g.
\]
\[
\tau(M)=b_+-b_-={{\frac{1}{48\pi^2}}}\int_M|W^+|^2-|W^-|^2\;d^4V_g.
\]
where $W^\pm$ is the Weyl curvature tensor split into its self-dual and anti-self-dual parts, $Ric$ is the Ricci tensor and $S$ is the scalar curvature of $g'$.

From these and Theorem \ref{t:4}, the following can be proven
\vspace{0.1in}
\begin{Thm}\cite{GG22b}\label{t:5}
Let $(M,g')$ be a closed, conformally flat, scalar flat, neutral 4-manifold. If $g'$ admits a parallel isometric paracomplex structure, then 
\[
\tau(M)=0\qquad\qquad{\mbox{ and }}\qquad\qquad\chi(M)\geq 0.
\]

If, moreover, the Ricci tensor of $g'$ has negative norm $|Ric(g')|^2\leq 0$, then $M$ admits a flat Riemannian metric.
\end{Thm}
\vspace{0.1in}

On the other hand, Theorem \ref{t:4} can also be used on Riemannian Einstein 4-manifolds to find obstructions to parallel isometric paracomplex structures:

\vspace{0.1in}
\begin{Thm}\cite{GG22b}\label{t:6}
Let $(M,g)$ be a closed Riemannian Einstein 4-manifold. 

If $g$ admits a parallel isometric paracomplex structure, then $\tau(M)=0$. 
\end{Thm}
\vspace{0.1in}

The $K3$ 4-manifold, as well as the 4-manifolds ${\mathbb{ C}}P^2\#k\overline{\mathbb CP}^2$ for $k=3,5,7$, admit Riemannian Einstein metrics and isometric almost paracomplex structures, but, as a consequence of Theorem \ref{t:6}, these almost paracomplex structures cannot be parallel.

\vspace{0.1in}
\noindent{\bf Acknowledgements}:

Most of the work described in this paper was carried out in collaboration with Guillem Cobos, Nikos Georgiou and Wilhelm Klingenberg, with whom it has been a pleasure to learn. Thanks are due to Morgan Robson for assistance with the Figures. Any opinions expressed are entirely the author's.

\end{document}